\documentclass[12pt]{amsart}
\usepackage{amsmath}
\usepackage{amssymb}
\usepackage{amsthm}
\usepackage{amscd}
\usepackage{amsfonts}
\usepackage{graphicx}%
\usepackage{fancyhdr}
\usepackage{enumerate}
\usepackage{mathrsfs}
\usepackage{ulem}
\usepackage{hyperref}
\usepackage[all]{xy}

\theoremstyle{plain} \numberwithin{equation}{subsection}
\newtheorem{theorem}{Theorem}[section]

\newtheorem{lemma}[theorem]{Lemma}
\newtheorem{proposition}[theorem]{Proposition}

\theoremstyle{definition}
\newtheorem{remark}[theorem]{Remark}
\newtheorem{remarks}[theorem]{Remarks}

\newtheorem{definition}[theorem]{Definition}

\makeatletter
\def\blfootnote{\gdef\@thefnmark{}\@footnotetext}
\makeatother

\def\mf#1{{ \mathfrak{#1} }}
\def\ms#1{{ \mathscr{#1} }}
\def\mb#1{{ \mathbb{#1} }}
\def\mc#1{{ \mathcal{#1} }}
\def\tr#1{{ \textrm{#1} }}
\def\tb#1{{ \textbf{#1} }}
\def\isom{{ \, \stackrel \sim \longrightarrow \, }}
\def\P{{ \mb P }}
\def\id{{ \tr{id} }}
\def\C{{\mb C}}

\def\Fp{{ \mb F_p }}

\def\Ga{{ G^a }}
\def\Um{{ U^- }}
\def\Uma{{ U^{-,a} }}

\def\g{{ \mf g }}
\def\b{{ \mf b }}

\def\n{{ \mf n }}
\def\nm{{ \mf n^- }}

\def\L{{ \mc L }}
\def\huma{{ \hat U^{-,a} }}

\def\barg{{ \bar U(\g) }}

\def\barnm{{ \bar U(\nm) }}

\def\bargf#1{{ \bar U_{\leq #1}(\g) }}
\def\barnmf#1{{ \bar U_{\leq #1}(\nm) }}

\def\barnms#1{{ \bar U_{< #1}(\nm) }}

\def\dom{{ \Lambda^+ }}
\def\pos{{ \Delta^+ }}

\def\ho#1{{ H^0(#1) }}
\def\Ho#1{{ H^0 \big(  #1  \big) }}

\def\V#1{{ V(#1) }}
\def\Vstar#1{{ V(#1^*) }}
\def\vstar#1{{ v_{#1^*} }} 
\def\vastar#1{{ v_{#1^*}^a }}
\def\Vss#1{{ V(#1^*)^* }} 
\def\Va#1{{ V^a(#1) }}
\def\VA#1{{ V^a \big( #1 \big) }}
\def\Vastar#1{{ V^a(#1^*) }}
\def\Vass#1{{ V^a(#1^*)^* }}
\def\vf#1#2{{ V_{\leq #1}(#2) }}
\def\vs#1#2{{ V_{< #1}(#2) }}
\def\hof#1#2{{ H^0_{\geq #1}(#2) }}

\def\Hof#1#2{{ H^0_{\geq #1}\big( #2 \big) }}

\def\hoa#1{{ H^{0, \, a}(#1) }}
\def\Hoa#1{{ H^{0, \, a}(#1) }}

\def\R#1{{ R_{#1} }}
\def\Ra#1{{ R^a_{#1} }}
\def\Sa#1{{ S^a_{#1} }}
\def\gr{{ \tr{gr} }}

\def\kum{{ k[U^-] }}
\def\kuma{{ k[U^{-,a}] }}
\def\im{{ I_- }}
\def\imp#1{{ I_-^{#1} }}

\def\Isec#1#2{{ \Gamma\big(  \flag, \, \mc I^{#1} \otimes \L(#2)  \big) }}

\def\hynm{{ \tr{hy}(\nm) }}
\def\hynmg#1{{ \tr{hy}_{#1}(\nm) }}
\def\F#1#2{{ F_{#1}^{(#2)} }}
\def\f#1#2{{ f_{#1}^{(#2)} }}

\def\ve#1{{ \varepsilon_{\beta_{#1}} }}
\def\RI#1#2{{ I_{#1, #2} }}
\def\proj{{ \tr{Proj} }}
\def\ts{{ \widetilde \sigma }}
\def\loc{{ \mc O_{X,x} }}

\def\cone{{ \mc T_e }}

\def\Fo{{F_0}}
\def\fo{{ f_0 }}
\def\shom{{ \ms H \!om }} 
\def\flag{{G/B}}
\def\flaga#1{{ \mc F^a_{#1} }}
\def\flagastar#1{{ \mc F^a_{#1^*} }}
\def\can{{ \omega_\flag^{1-p} }}
\def\tprho{{ 2(p-1)\rho }}
\def\pN{{ (p-1)N }}
\def\s{{ \mf s }}
\def\etagr#1{{ \eta^{\gr}_#1 }}
\def\Pa#1{{ \P^a_{#1} }}
\def\Pastar#1{{ \P^a_{#1^*} }}
\def\symn{{ S(\n) }}

\def\ofa#1#2{{ \mc O_{\flagastar #1}(#2) }}
\def\opa#1#2{{ \mc O_{\Pastar #1}(#2) }}

\def\s{{ \mf s }}
\def\r{{ \mf r }}
\def\vp{{ v^a_{\lambda} }}
\def\vps{{ v^a_{\lambda^*} }}
\def\div#1{{ \mc O(#1) }}

\def\linv{{ L^{-1} }}

\begin{document}

\author{Chuck Hague}
\email{hague@math.udel.edu}
\address{Department of Mathematical Sciences \\
University of Delaware \\
501 Ewing Hall \\
Newark, DE 19716}
\title{Degenerate coordinate rings of flag varieties and Frobenius splitting}
\begin{abstract} Recently E. Feigin introduced the $\mathbb G_a^N$-degenerations of semisimple algebraic groups and their associated degenerate flag varieties. It has been shown by Feigin, Finkelberg, and Littelmann that the degenerate flag varieties in types $A_n$ and $C_n$ are Frobenius split. In this paper we construct an associated degeneration of homogeneous coordinate rings of classical flag varieties in all types and show that these rings are Frobenius split in most types. It follows that the degenerate flag varieties of types $A_n$, $C_n$ and $G_2$ are Frobenius split. In particular we obtain an alternate proof of splitting in types $A_n$ and $C_n$; the case $G_2$ was not previously known. We also give a representation-theoretic condition on PBW-graded versions of Weyl modules which is equivalent to the existence of a Frobenius splitting of the classical flag variety that maximally compatibly splits the identity.
\keywords{Frobenius splitting \and degenerate flag varieties \and PBW filtration}
\end{abstract}
\maketitle

\setcounter{tocdepth}{2} 
\tableofcontents



\section{Introduction}

Let $k$ be an algebraically closed field of positive characteristic $p$ and let $G$ be a semisimple algebraic group over $k$. In a series of papers (\cite{FeDeg}, \cite{FPBW}, \cite{FeGa}, \cite{FeOrb}, \cite{FFA}, \cite{FFLC}, \cite{FFLPBWA}, \cite{FFLPBWC}, \cite{FFLPBWZ}), Feigin, Finkelberg, Fourier and Littelmann have investigated the PBW filtration on Weyl modules for $G$, the associated degenerate group $\Ga$, and related projective varieties called degenerate flag varieties (see \S\ref{subsub:Ga}, \S\ref{sub:PBW} and \S\ref{sub:flaga} below for precise definitions). The $G_2$ case has been recently considered in \cite{Gor} (see also \S3.3 in \cite{FeOrb}).

In more detail, the PBW filtration on a Weyl module $\V \lambda$ of highest weight $\lambda$ gives rise to an associated graded $\Ga$-module $\Va \lambda$ with dual module $ \hoa{ -w_0\lambda } $, where $ w_0 $ is the longest element of the Weyl group. The modules $\hoa \lambda$ for dominant $\lambda$ are degenerate analogs of the standard induced modules $\ho \lambda$ for $G$ and it is natural to consider the degenerate version
\begin{equation}
\Ra \lambda := \bigoplus_{n \geq 0} \hoa {n\lambda}
\end{equation}
of the standard section ring $\bigoplus_{n \geq 0} \ho {n\lambda}$ for the flag variety $\flag$.

Unlike the standard section ring for $\flag$, however, it is not a priori clear that $\Ra \lambda$ has a natural ring structure. It is shown in \cite{FeOrb}, \cite{FFLPBWZ} that in types $A_n$, $C_n$ and $G_2$ there is a $\Ga$-equivariant multiplication map $ m : \hoa \lambda \otimes \hoa \mu \to \hoa{\lambda + \mu} $ for all dominant $\mu, \lambda$ coming from a natural comultiplication map $ \Va{\lambda + \mu} \to \Va \lambda \otimes \Va \mu $. In \S\ref{sec:Ra} we show that there is a natural $\Ga$-equivariant multiplication $m$ in all types and hence $\Ra \lambda$ is a $\Ga$-algebra. Consider the decreasing filtration on $\ho \lambda$ coming from vanishing degree at the identity element in $\flag$; then we show in Lemma \ref{lem:ib and hof} that $\hoa \lambda$ is the associated graded module for this filtration. The existence of the multiplication map $m$ immediately follows since vanishing degree is multiplicative. Although one can show (cf Proposition \ref{pr:Ra and Sa}) that $\Ra \lambda$ is a domain, it is not clear if it is Noetherian in general although this is known in types $A_n$, $C_n$ and $G_2$. This lack of well-behavedness is one of the main technical issues in the last part of this paper and requires some delicate handling.

Next, we show in \S\ref{sec:splitting of Ra} (cf Theorem \ref{th:Frobenius splitting of Ra}) that if $\flag$ has a Frobenius splitting that maximally compatibly splits the identity then the ring $\Ra \lambda$ is Frobenius split. In particular, $\Ra \lambda$ is Frobenius split in the classical types and in type $G_2$. We then show in Proposition \ref{pr:fo} that the existence of such a splitting of $\flag$ is equivalent to an interesting representation-theoretic condition on the degenerate Weyl module $\Va \tprho$, where $\rho$ is the half-sum of the positive roots of $G$.

In the last section \S\ref{sec:geometry} we turn our attention to the degenerate flag varieties. For dominant $\lambda$ we consider the section ring $\Sa \lambda$ on an associated degenerate flag variety. It is not clear in general whether this ring is isomorphic to $\Ra \lambda$, although this is known to be true in types $A_n$, $C_n$ and $G_2$ by \cite{FeGa} and \cite{FeOrb}. In these cases the Frobenius splitting of the degenerate flag variety follows from the splitting of the ring $\Ra \lambda$, cf Theorem \ref{th:ACG splitting}. In particular, we obtain an alternate proof of Frobenius splitting results for degenerate flag varieties in types $A_n$ and $C_n$, which was originally proved in \cite{FFA} and \cite{FFLC}; that the degenerate flag varieties of type $G_2$ are Frobenius split was previously unknown. 

We then compare the rings $\Ra \lambda$ and $\Sa \lambda$ in all types. We formulate conditions under which $\proj (\Ra \lambda)$ is isomorphic to the associated degenerate flag variety, a weaker condition than having an isomorphism $\Ra \lambda \cong \Sa \lambda$. These conditions imply that the associated degenerate flag variety is Frobenius split, cf Theorem \ref{th:splitting of flaga}.

Also, I would like to thank Evgeny Feigin and Shrawan Kumar for helpful conversations and the anonymous referee for carefully reading the paper and pointing out typos and areas for improvement.

\section{Background}

\subsection{$G$ and $\Ga$}

\subsubsection{$G$}
Let $G$ be as above. Let $B \subseteq G$ be Borel subgroup and let $U \subseteq B$ be its unipotent radical. Let $B^- \subseteq G$ be the opposite Borel subgroup to $B$ and let $U^- \subseteq B^-$ be its unipotent radical. Let $T \subseteq B$ be a maximal torus. Set $\g = \tr{Lie}(G)$, $\n = \tr{Lie}(U)$ and $\nm = \tr{Lie}(U^-)$. Denote by $ \barg $ and $\barnm$ the hyperalgebras of $G$ and $U^-$, respectively. Let $\bargf n$ denote the degree filtration on $\barg$. This filtration restricts to a degree filtration $\barnmf n$ on $\barnm$. 

Let $\Delta \subseteq \Lambda$ denote the roots of $G$ and let $\pos \subseteq \Delta$ denote the positive roots corresponding to our choice of Borel $B$. Set $N := |\pos|$ and denote by $\rho$ the half-sum of the positive roots. Let $\Lambda$ denote the weight lattice of $T$ and let $\dom \subseteq \Lambda$ be the set of dominant weights. For any $\lambda \in \Lambda$ we have the one-dimensional $B$-module $\chi_\lambda$ with weight $\lambda$.

For $\lambda \in \dom$ we have the Weyl module $V(\lambda)$ with highest weight $\lambda$ and the induced module $\ho \lambda$ of highest weight $\lambda$. Set $\lambda^* := -w_0 \lambda \in \dom$, where $w_0$ is the longest element of the Weyl group of $G$; then we have $\ho \lambda \cong \Vss \lambda$ and there is a $G$-equivariant duality pairing
\begin{equation} \label{eq:eta}
\eta_\lambda : \ho \lambda \otimes \Vstar \lambda \to k.
\end{equation} For each $\lambda \in \dom$ choose a nonzero highest-weight element $v_\lambda \in \V \lambda$.

Fix a Chevalley basis $ \{ F_{-\beta}, E_\beta : \beta \in \pos \}  \subseteq \g $. Then $\barnm$ is generated as a $k$-algebra by the divided-power elements $\F \beta n \in \barnm$ for $\beta \in \pos$ and $n \geq 0$. Denote by $\hynm$ the hyperalgebra of $\nm$, where we consider $\nm$ as an algebraic group isomorphic to a product of copies of $ \mb G_a $. By the PBW theorem there is a $k$-algebra isomorphism
\begin{equation}
\gr \, \barnm \cong \hynm.
\end{equation}
Denote by $ \f \beta n \in \hynm $ the image of the basis element $\F \beta n$ under the projection $\barnmf n \twoheadrightarrow \hynmg n$. Then $\hynm$ is a divided-power polynomial ring over $k$ on the generators $ \{ \f \beta n : \beta \in \pos, n > 0 \} $. Further, $ \hynm $ has a natural $B$-module algebra structure coming from the $B$-module structure on $ \nm \cong \g / \b $.

\subsubsection{$\Ga$} \label{subsub:Ga}

Following \cite{FFLPBWZ}, we construct a degenerate form $\Ga$ of $G$ by abelianizing the negative part $U^-$ of $G$ as follows. First, $U^-$ has a filtration by normal subgroups
\begin{equation}
U^-_s := \prod_{\tr{ht} (\beta) \geq s} U_{-\beta}.
\end{equation}
Set
\begin{equation}
\Uma := \prod_{s \geq 1} U^-_s / U^-_{s+1}.
\end{equation}
Then $\Uma$ is a commutative unipotent group with a $B$-action induced by the conjugation action on $\Um$. In particular, $\Uma$ is isomorphic to a product of copies of the additive group $ \mb G_m $. We now set
\begin{equation}
\Ga := B \ltimes \Uma.
\end{equation}
Remark that $\hynm$ is $\Ga$-equivariantly isomorphic to the hyperalgebra of $ \Uma $.

\subsection{The PBW filtration and the dual filtration} \label{sub:PBW}

Following \cite{FFLPBWZ}, for $\lambda \in \dom$ define an increasing filtration on the cyclic $\barnm$-module $V(\lambda)$ by
\begin{equation}
\vf n \lambda := \barnmf n . v_\lambda
\end{equation}
(recall that $v_\lambda \in \V \lambda$ is a nonzero highest weight vector). We call this filtration the \tb{PBW filtration} on $\V \lambda$. This filtration is $B$-stable and gives rise to an associated graded $B$-module
\begin{equation}
\Va \lambda := \gr \,\V \lambda.
\end{equation}
Dually, replacing $\lambda$ by $\lambda^*$ we obtain a decreasing $B$-stable filtration on $ \ho \lambda $ by
\begin{equation} \label{eq:hof}
\hof n \lambda := \vs n {\lambda^*}^\perp \subseteq \Vss \lambda = \ho \lambda.
\end{equation}
Explicitly, we have
\begin{equation} \label{eq:hof explicitly}
\hof n \lambda = \big \{   v \in \ho \lambda : \eta_\lambda\big(  v \otimes \barnms n.\vstar \lambda  \big) = 0  \big \}.
\end{equation}
We now obtain an associated graded $B$-module
\begin{equation}
\hoa \lambda := \gr \, \ho \lambda.
\end{equation}
By \S2, \S3 in \cite{FFLPBWZ} there is a natural $\hynm$-module structure on $ \Va \lambda $ coming from the $\barnm$-module structure on $\V \lambda$. It follows that $\Va \lambda$ is a cyclic $\hynm$-module and the $\hynm$-module structure integrates to a $\Uma$-module structure on $\Va \lambda$. We now obtain a natural $\Ga$-module structure on $\Va \lambda$ which agrees with the $B$-module structure described above. Dually, we obtain a $\Ga$-module structure on $\hoa \lambda$ and the nondegenerate $G$-equivariant duality pairing $\eta_\lambda$ of (\ref{eq:eta}) induces a graded nondegenerate $\Ga$-equivariant pairing
\begin{equation} \label{eq:etagr}
\etagr \lambda : \hoa \lambda \otimes \Vastar \lambda \to k.
\end{equation}

In the sequel we will denote the graded components of $ \Va \lambda $ and $ \hoa \lambda $ by $\Va \lambda_m$ and $\hoa \lambda_m$. Also, the highest-weight spaces in $\V \lambda$ and $\Va \lambda$ are naturally isomorphic; let $ v^a_\lambda \in \Va \lambda $ be the image of the highest-weight vector $v_\lambda \in \V \lambda$ under this isomorphism.

\section{The degenerate coordinate ring $\Ra \lambda$} \label{sec:Ra}

\subsection{The filtration $\Isec n \lambda$} \label{sub:Isec}

For $\lambda \in \dom$ let $\mc L(\lambda)$ denote the line bundle on $G/B$ with fiber $\chi_{w_0\lambda}$ (recall that for $\mu \in \Lambda$, $\chi_\mu$ denotes the 1-dimensional $B$-module of weight $\mu$). Then we have
\begin{equation} \label{eq:ind isomorphisms}
\Gamma\big(  G/B, \mc L(\lambda)  \big) \cong \ho \lambda.
\end{equation}
Let $\mc I \subseteq \mc O_{G/B}$ be the ideal sheaf of identity $eB \in \flag$. Since the identity is a $B$-stable point, $\Isec n \lambda$ is a $B$-submodule of $ \Gamma\big(  G/B, \mc L(\lambda)  \big) $ for all $n \geq 0$.

Recall the $B$-stable filtration $\hof n \lambda$ on $\ho \lambda$ from (\ref{eq:hof}). The following important lemma is motivated by \S4 of \cite{KuNil}.
\begin{lemma} \label{lem:ib and hof}
The isomorphism (\ref{eq:ind isomorphisms}) restricts to a $B$-equivariant isomorphism
$$ \Isec n \lambda \cong \hof n \lambda$$
for all $n \geq 0$.
\end{lemma}

\begin{proof}
Identify $\Um$ with the big cell $\Um B \subseteq \flag$. Then there is a section restriction inclusion
\begin{subequations}
\begin{equation} \label{eq:a1}
a : \ho \lambda \hookrightarrow \kum
\end{equation}
given by
\begin{equation} \label{eq:a2}
a(v)(g) = \eta_\lambda\big(  v \otimes  g.\vstar \lambda \big)
\end{equation}
\end{subequations}
for all $v \in \ho \lambda$ and $g \in U^-$. For the rest of the proof we will consider $\ho \lambda = \Gamma( \flag, \L(\lambda) )$ as a subspace of $\kum$ via this inclusion.

Let $ \im \subseteq \kum $ be the ideal of the identity. Then the restriction of the ideal sheaf $\mc I$ to $\Um$ is $\im$. Thus, considering $ \Isec n \lambda $ as a subspace of $ \ho \lambda $ inside of $\kum$, we have
\begin{equation*}
\Isec n \lambda = \imp n \cap \ho \lambda.
\end{equation*}
Hence it suffices to show that
\begin{equation} \label{eq:ib and hof}
\hof n \lambda = \ho \lambda \cap \imp n
\end{equation}
for all $n \geq 0$. This now follows from the usual technique of integrating the $\barnm$-action on $\Vstar \lambda$ to the $U^-$-action. For completeness we sketch the details here.

Fix an ordering $ \beta_1, \ldots, \beta_N $ of $\pos$. For each $1 \leq i \leq N$ let
$$ \ve i : \mb G_m \isom U_{-\beta_i} \subseteq U^- $$
be an algebraic group isomorphism such that $ d \ve i = F_{\beta_i} $. Then for all $ c_1, \ldots, c_N \in k $ we have
\begin{equation}
\ve 1(c_1) \cdots \ve N(c_N).\vstar \lambda = \sum c_1^{n_1} \cdots c_N^{n_N} \cdot  \F{\beta_1}{n_1} \ldots \F{\beta_N}{n_N}.\vstar \lambda,
\end{equation}
where this sum ranges over all tuples $(n_1, \ldots, n_N) \in \mb Z_+^N$. Thus by (\ref{eq:a2}) we have
\begin{equation*}
a(v)\big(  \ve 1(c_1) \cdots \ve N(c_N)  \big) = \sum c_1^{n_1} \cdots c_N^{n_N} \cdot \eta_\lambda\big(  v \otimes \F{\beta_1}{n_1} \ldots \F{\beta_N}{n_N}.\vstar \lambda  \big).
\end{equation*}
Comparing with the explicit description (\ref{eq:hof explicitly}) of $\hof n \lambda$ we obtain the equality (\ref{eq:ib and hof}).
\end{proof}

\subsection{Standard and degenerate coordinate rings}

Choose $\lambda, \mu \in \dom$. By Lemma \ref{lem:ib and hof} the multiplication map $ \ho{\lambda} \otimes \ho{\mu} \twoheadrightarrow \Ho{\lambda + \mu} $ is filtration-preserving; that is, this multiplication map restricts to a $B$-equivariant multiplication
\begin{equation} \label{eq:ho mult}
\hof i{\lambda} \otimes \hof j{\mu} \to \Hof{i+j}{\lambda+\mu}
\end{equation}
for all $i,j \geq 0$. Taking the associated graded modules we have an induced multiplication morphism
\begin{subequations} 
\begin{equation} \label{eq:hoa mult}
\hoa {\lambda} \otimes \hoa{\mu} \to \Hoa{\lambda+\mu}.
\end{equation}
Dually, we have a $B$-equivariant comultiplication morphism
\begin{equation} \label{eq:Va comult}
\VA {\lambda+\mu} \to \Va{\lambda} \otimes \Va{\mu}.
\end{equation}
\end{subequations}
By the results in \cite{Gor} (cf \S3.3 in \cite{FeOrb}) and Proposition 8.1 in \cite{FFLPBWZ}, if $G$ is of type $A$, $C$ or $G_2$ then the multiplication morphism (\ref{eq:hoa mult}) is surjective for all $i, j$ (and, equivalently, \ref{eq:Va comult} is injective). It is not known if (\ref{eq:hoa mult}) is surjective in the other types.

For $\lambda \in \dom$ set
\begin{equation}
\R \lambda := \bigoplus_{n \geq 0} \ho {n\lambda},
\end{equation}
the homogeneous coordinate ring of a partial flag variety $G/P$ for $G$ under the embedding $G/P \hookrightarrow \P (\Vstar\lambda)$. Set
\begin{equation}
\Ra \lambda := \bigoplus_{n \geq 0} \hoa{n\lambda},
\end{equation}
a degenerate version of the coordinate ring $\R \lambda$.

We have the following equivalent description of $\Ra \lambda$. Following Lemma \ref{lem:ib and hof}, for each $m \geq 0$ set
\begin{equation} \label{eq:I lambda}
\RI \lambda m := \bigoplus_{n \geq 0} \hof m{n\lambda} = \bigoplus_{n \geq 0} \Gamma\big(  G/B, \, \mc I^m \otimes \L(n\lambda)  \big) ,
\end{equation}
a $B$-stable ideal of $\R \lambda$. Then the ideals $\RI \lambda m$ form a decreasing filtration of $\R \lambda$ such that $ \RI \lambda m \cdot \RI \lambda k \subseteq \RI \lambda{m+k} $ for all $ m, k \geq 0 $ and we have
\begin{equation} \label{eq:gr R}
\gr \R \lambda = \Ra \lambda.
\end{equation}

\section{Frobenius splitting of $\Ra \lambda$} \label{sec:splitting of Ra}

\subsection{Background on Frobenius splitting}

We give a quick recollection of basic Frobenius splitting results (see \cite{BK} for further details). Let $ A $ be a $k$-algebra. We say that an $\Fp$-linear morphism $\sigma : A \to A$ is \tb{Frobenius-linear} if $\sigma(f^p g) = f \cdot \sigma g$ for all $ f, g \in A $. We say that $A$ is \tb{Frobenius split} if there is a Frobenius-linear endomorphism $\sigma$ such that $\sigma(1) = 1$.

Sheafifying this idea, we obtain the idea of a Frobenius split $k$-scheme. Let $X$ be a $k$-scheme and let $F : X \to X$ be the (absolute) Frobenius morphism, i.e. the morphism which is the identity on points and the $p^{th}$ power map on functions. Then we say that $X$ is \tb{Frobenius split} if there is a sheaf map $\sigma \in \shom_{\mc O_X}( F_* \mc O_X, \mc O_X )$ such that $\sigma(1) = 1$.

A Frobenius splitting of $X$ lifts to a splitting on sections of invertible sheaves in the following way. Given an invertible sheaf $\L$ on a projective $k$-variety $X$ set
\begin{equation} \label{eq:R(L)}
R(\L) := \bigoplus_{n \geq 0} \Gamma( X, \L^n ),
\end{equation}
the section ring of $\L$. Let $\ts : F_* \mc O_X \to \mc O_X$ be a splitting of $X$. Then in this case there is a natural graded splitting $\sigma$ of $R(\L)$ induced by $\ts$, where by a \tb{graded splitting} we mean a splitting that sends graded degree $pm$ to graded degree $m$ and kills all other degrees that are not divisible by $p$.

We can explicitly describe $\sigma$ as follows (cf Definition 4.1.12 in \cite{BK}). There is an $\mc O_X$-linear isomorphism $ F^* \mc M \cong \mc M^p $ for any invertible sheaf $\mc M$ on $X$ and hence by Lemma 1.2.6 in \cite{BK} a (non-$k$-linear!) isomorphism
\begin{equation}
\Gamma( X, \L^{mp} ) \cong \Gamma(X, F_* F^* \L^m) \cong \Gamma(X, \L^m \otimes_{\mc O_X} F_* \mc O_X)
\end{equation}
for all $m \geq 0$. Then the map
$$ \id \otimes \ts : \L^m \otimes_{\mc O_X} F_* \mc O_X \to \L^m \otimes_{\mc O_X} \mc O_X = \L^m $$
induces a map
$$  \Gamma( X, \L^{pm} ) \to \Gamma(X, \L^m)  $$
on global sections. Taking the direct sum of these morphisms and extending by zero on graded degrees $k$ with $p \nmid k$, we obtain a morphism $ \sigma : R(\L) \to R(\L) $ which can be checked to be Frobenius-linear (cf the proof of Lemma 4.1.13 in \cite{BK}). It follows that $\sigma$ splits the $p^{th}$ power morphism on $R(\L)$ and hence is a graded splitting of $R(\L)$.

We also recall the following essential result from \cite{BK}, \cite{LTmax}.

\begin{lemma}[\cite{BK}, Lemma 1.1.14 and \cite{LTmax}, Lemma 2.7] \label{lem:projective splitting}
Let $ A $ be a Noetherian graded $k$-algebra with $A_0 = k$ and set $X := \proj \, A$. Then $A$ is a Frobenius split ring if and only if $X$ is Frobenius split.
\end{lemma}

\subsection{Frobenius splitting of $\Ra \lambda$}

\begin{definition}
Let $X$ be a nonsingular variety and let $Y \subseteq X$ be a closed nonsingular variety with ideal sheaf $\mc I \subseteq \mc O(X)$. Following Definition 2.2 in \cite{LTmax}, we say that $Y$ is \tb{maximally compatibly split} if there is a Frobenius splitting $ \sigma : F_*\mc O_X \to \mc O_X $ such that $ \sigma\big(  \mc I^{pm + 1}  \big) \subseteq \mc I^{m+1} $ for all $m \geq 0$. (Remark that the equivalent condition proved in Lemma 2.12 of \cite{LTmax} is frequently given as the definition of a maximal compatible splitting, cf exercises 1.3.E.12 and 13 in \cite{BK}.)
\end{definition}

Recall the definition of the ring $R(\L)$ from (\ref{eq:R(L)}) above. 

\begin{proposition} \label{pr:max split section ring}
Let $X$ be a smooth projective $k$-variety and let $\mc L$ be an invertible sheaf on $X$. Let $Y \subseteq X$ be a smooth closed subvariety with ideal sheaf $\mc I$ and assume that $Y$ is maximally compatibly split by a splitting $ \ts: F_* \mc O_X \to \mc O_X $ of $X$. For $m \geq 0$ set $ J_m := \bigoplus_{n \geq 0} \Gamma( X, \mc I^m \otimes \L^n ) $, an ideal of $ R(\L) $. Let $ \sigma : R(\L) \to R(\L) $ be the Frobenius splitting of $R(\L)$ obtained from $\ts$ by the discussion above. Then $\sigma (J_{pm+1}) \subseteq J_{m+1}$ for all $m \geq 0$. 
\end{proposition}

\begin{proof}
We may assume without loss of generality that $X$ is irreducible. Fix $x \in Y$ and choose $n \geq 0$. Choose a trivialization $t : \Gamma(X, \L^n) \hookrightarrow \loc$. This trivialization induces a trivialization $t^p : \Gamma( X, \L^{pn} ) \hookrightarrow \loc $ such that $t^p(s^p) = t(s)^p$ for all $s \in \Gamma(X, \L^n)$.

Let $K \subseteq \loc$ be the local ideal defining $Y$ at $x$; then $K$ is the stalk of $\mc I$ at $x$. Since $Y$ is smooth, for all $j \geq 0$ we have that $K^j$ is the ideal of $\loc$ consisting of elements that vanish along $Y$ to order at least $j$. Denote by $ \ts_x : \loc \to \loc $ the splitting of $ \loc $ induced by $\ts$. Then we have a commutative diagram
\begin{equation}
\xymatrix{
\Gamma(X, \L^{pn}) \ar@{^(->}[r]^{ \hspace{.2in} t^p} \ar[d]_\sigma & \loc \ar[d]^{\ts_x} \\
\Gamma(X, \L^n) \ar@{^(->}[r]_{ \hspace{.2in} t } & \loc
}
\end{equation}
which restricts to a commutative diagram
\begin{equation}
\xymatrix{
\Gamma(X, \mc I^{pm+1} \otimes \L^{pn}) \ar@{^(->}[r]^{ \hspace{.4in} t^p} \ar[d]_\sigma & K^{pm+1} \ar[d]^{\ts_x} \\
\Gamma(X, \L^n) \ar@{^(->}[r]_{ \hspace{.15in} t } & \loc
}
\end{equation}
for all $m \geq 0$. Now, we have $ \ts_x\big(  K^{pm+1}  \big) \subseteq K^{m+1} $ by the definition of a maximal compatible splitting, which implies
$$  \sigma\big(  \Gamma(X, \mc I^{pm+1} \otimes \L^{pn})  \big) \subseteq t^{-1} K^{m+1} = \Gamma(X, \mc I^{m+1} \otimes \L^{pn}) .$$
Hence $\sigma(J_{pm+1}) \subseteq J_{m+1}$ for all $m \geq 0$ as desired.
\end{proof}

By \cite{LMP}, we know that if each simple component of $G$ is of classical type or type $G_2$ then there is a splitting of $G/B$ that maximally compatibly splits the identity. As a result of (\ref{eq:I lambda}), (\ref{eq:gr R}) and Proposition \ref{pr:max split section ring} we now have the

\begin{theorem} \label{th:Frobenius splitting of Ra}
Assume that there is a Frobenius splitting of $G/B$ that maximally compatibly splits the identity. Then for all $\lambda \in \dom$ the degenerate coordinate ring $\Ra \lambda$ is Frobenius split. In particular, if $G$ is of classical type or of type $G_2$ then $\Ra \lambda$ is Frobenius split.
\end{theorem}

\subsection{A representation-theoretic analog of maximal compatible splittings}

We now derive a representation-theoretic statement which is equivalent to the existence of a splitting of $\flag$ that maximally compatibly splits the identity.

Set $\Fo := \displaystyle \prod_{\beta \in \pos} \F \beta{p-1} \in \barnm$, a so-called norm form for the hyperalgebra of the first Frobenius kernel of $U^-$. By Lemma 6.6 and Proposition 6.7 in \cite{Hab80}, $\Fo$ is independent of the order of roots in the product and is central in $\barnm$. Also set $\fo := \displaystyle \prod_{\beta \in \pos} \f \beta{p-1} \in \hynm$.

Recall that $\rho$ denotes the half-sum of the positive roots. Also recall the definition of the highest-weight vector $ v^a_\tprho \in \Va \tprho $ from \S\ref{sub:PBW}. We now have the following representation-theoretic interpretation of maximal compatible splittings.

\begin{proposition} \label{pr:fo}
There is a Frobenius splitting of $\flag$ that maximally compatibly splits the identity if and only if $\fo.v^a_\tprho \neq 0$ in $\Va \tprho$.
\end{proposition}

\begin{proof}
We first recall the following general geometric setting for Frobenius splittings of $G/B$ (cf \cite{BK}). For any smooth variety $X$ there is an isomorphism
\begin{equation}
\shom_{\mc O_X} \big(  F_* \mc O_X , \mc O_X \big) \cong \Gamma \big( X, \omega_X^{1-p} \big)
\end{equation}
and an evaluation morphism (often called a trace morphism) $ev : \Gamma \big( X, \omega_X^{1-p} \big) \to k$ which detects splittings, in the following sense: a section $s \in \Gamma \big( X, \omega_X^{1-p} \big)$ is a splitting if and only if $ev(s) = 1$. In the case $X = \flag$ we have $ \can = \L\big(  \tprho  \big) $. By rescaling $v_\tprho$ by a nonzero scalar if necessary, the evaluation morphism
$$ev : \ho \tprho \to k$$
is given by
\begin{equation} \label{eq:ev}
ev(s) = \eta_\tprho\big(  s \otimes \Fo.v_\tprho  \big)
\end{equation}
for all $s \in \ho \tprho$.

By Lemma \ref{lem:ib and hof}, $ \hof m \tprho $ is the subspace of $ \ho \tprho $ consisting of elements that vanish to multiplicity at least $m$ at the identity in $\flag$. It follows by Lemma 2.12 in \cite{LTmax} that a splitting section $s \in \ho \tprho$ maximally compatibly splits the identity if and only if $s \in \hof \pN \tprho$.

On the other hand, recall that $N = |\pos|$. Since $ \Fo $ is an element of $\barnm$ of filtration degree $\pN$ it follows that $\Fo.v_\tprho \in \vf \pN \tprho$. Further, $\fo$ is the image of $\Fo$ under the projection $ \barnmf \pN \twoheadrightarrow \hynmg \pN $. Thus we have
\begin{equation}
\fo.v^a_\tprho \neq 0 \tr{ if and only if } \Fo.v_\tprho \notin \vs \pN \tprho.
\end{equation}

Let $s \in \ho \tprho$ be a section that compatibly splits the identity with maximal multiplicity. Then by the above discussion
$$s \in \hof \pN \tprho = \vs \pN \tprho^\perp $$
and
$$  \eta\big(  s \otimes \Fo.v_\tprho  \big) = 1 . $$
This implies $ \Fo.v_\tprho \notin \vs \pN \tprho $ and hence $ \fo.v^a_\tprho \neq 0 $.

Conversely, assume that $\fo.v^a_\tprho \neq 0$. Recall the graded nondegenerate pairing
\begin{equation}
\etagr \tprho : \hoa \tprho \otimes \Va \tprho \to k
\end{equation}
from (\ref{eq:etagr}). Choose $s' \in \hoa \tprho_\pN$ such that
$$ \etagr \tprho( s' \otimes \fo.v^a_\tprho ) = 1  $$
and let $s \in \hof \pN \tprho$ be any lift of $s'$ under the projection
$$ \hof \pN \tprho \twoheadrightarrow \hoa \tprho_\pN . $$
Then
$$  \eta_\tprho\big(  s \otimes \Fo.v_\tprho  \big) = \etagr \tprho\big(  s' \otimes \fo.v^a_\tprho  \big) = 1 $$
so that $s$ is a splitting. Since $s \in \hof \pN \tprho$, $s$ splits the identity with maximal multiplicity as desired.
\end{proof}

\begin{remark}
Assume that $G$ is of type $A_n$, $C_n$ or $G_2$. It follows from the combinatorial description of bases for $\Va \lambda$ given in \cite{FeOrb}, \cite{FFLPBWA}, \cite{FFLPBWC}, \cite{FFLPBWZ}, \cite{Gor} that we have $\fo.v^a_\tprho \neq 0$ in $\Va \tprho$. Thus in those types Proposition \ref{pr:fo} gives an alternate proof that the identity is maximally compatibly split in $\flag$; as noted above, this was initially proved in \cite{LMP} in the classical types and in type $G_2$.
\end{remark}

\section{Geometry} \label{sec:geometry}

\subsection{The degenerate flag variety $\flaga \lambda$} \label{sub:flaga}

For $\lambda \in \dom$ set $\Pa \lambda := \P\big(  \Va \lambda  \big)$ and let $[\vp] \in \Pa \lambda$ be the class of the highest weight vector $\vp \in \Va \lambda$. Following \cite{FeGa} we define the degenerate (partial) flag variety to be
\begin{equation}
\flaga \lambda := \overline{ \Ga.[\vp] } = \overline{ \Uma.[\vp] } \subseteq \Pa \lambda.
\end{equation}
This construction is analogous to the construction of the classical (partial) flag variety as an orbit in $\P\big(  \V \lambda  \big)$. However, unlike in the classical case, $ \flaga \lambda $ is not a single $\Ga$-orbit and it is not smooth in general.

It is shown in \cite{FeOrb}, Corollary 3.4 that in types $A_n$, $C_n$ and $G_2$ the ring $ \Ra{\lambda^*} $ is the homogeneous coordinate ring of $\flaga \lambda$. By Lemma \ref{lem:projective splitting} and Theorem \ref{th:Frobenius splitting of Ra} we immediately have:

\begin{theorem} \label{th:ACG splitting}
Assume that $G$ is of type $A_n$, $C_n$ or $G_2$. Then for any $\lambda \in \dom$ the degenerate flag variety $ \flaga \lambda $ is Frobenius split.
\end{theorem}  

\begin{remark}
As noted earlier, the fact that $\flaga \lambda$ is Frobenius split in types $A_n$ and $C_n$ was initially proved in \cite{FFA} and \cite{FFLC} using a different technique.
\end{remark}

\subsection{Splitting results for $\flaga \lambda$}

We now turn our attention to a comparison of the ring $\Ra \lambda$ with the homogeneous coordinate ring of an associated degenerate flag variety (called $\Sa \lambda$ below) in any type. Although it is reasonable to hope that these rings are isomorphic in all types it is not clear how to show this in general, so a coarser analysis is necessary. 

\subsubsection{The section ring $\Sa \lambda$}
Given $\lambda \in \dom$, for now we will consider $ \flagastar \lambda $ rather than $ \flaga \lambda $. Although it is slightly cumbersome we do this because we want to focus on $\hoa \lambda = \Vass \lambda$.

For $n \geq 0$ consider the very ample invertible sheaf $\div n = \ofa \lambda n$ on $\flagastar \lambda$ induced by the inclusion $\flagastar \lambda \hookrightarrow \Pastar \lambda$. Set
\begin{equation}
\Sa \lambda := \bigoplus_{n \geq 0} \Gamma \big(  \flagastar \lambda, \div n  \big).
\end{equation}
Then $ \flagastar \lambda = \proj (\Sa \lambda) $.

\subsubsection{Comparison of $\Ra \lambda$ and $\Sa \lambda$} \label{subsub:Ra and Sa 1}

Choose $\lambda \in \dom$ and let
$$i : \hoa \lambda = \Vss \lambda = \Gamma\big(  \Pastar \lambda, \opa \lambda 1 \big) \to \Gamma\big(  \flagastar \lambda, \ofa \lambda 1  \big)$$
be the $\Ga$-equivariant morphism given by section restriction. As noted in the proof of Theorem 5.2 in \cite{FFA}, by considering restriction of sections to the degenerate big cell $\Uma.[\vstar \lambda] \subseteq \flagastar \lambda$ we see that $i$ is injective. Indeed, this follows from the observation that the set $ \Uma.\vstar \lambda $ spans $\Vastar \lambda$.

The goal of this section is to prove the following crucial proposition which allows us to compare the rings $\Ra \lambda$ and $\Sa \lambda$.

\begin{proposition} \label{pr:Ra and Sa}
There are $\Uma$-equivariant algebra inclusions
\begin{equation}
\r : \Ra \lambda \hookrightarrow k[\Uma] \otimes k[t]
\end{equation}
and
\begin{equation}
\s : \Sa \lambda \hookrightarrow k[\Uma] \otimes k[t]
\end{equation}
such that the following diagram commutes:
\begin{equation} \label{eq:r and s}
\xymatrix{
& \Ra \lambda \ar@{^(->}[dr]^\r & \\
\hoa \lambda \ar@{^(->}[ur] \ar@{_(->}[dr]_i && k[\Uma] \otimes k[t] \\
& \Sa \lambda \ar@{_(->}[ur]_\s &
}
\end{equation}
(Here we put the trivial $\Uma$-algebra structure on $k[t]$.)
\end{proposition}

\begin{proof}
We split the proof into multiple parts.

(1) \tb{The map $\r$.}

This map is the extension of a special case of Corollary 4.3 and Theorem 4.4 in \cite{KuNil} to an arbitrary base field. Following \cite{KuNil} we may identify $\kuma \cong \symn$ with the tangent cone $\cone$ to $e \in \flag$. Under this identification, it follows from Corollary 4.3 and Theorem 4.4 of \cite{KuNil} (where we take the special case $v = w = e$) that there is a $\Uma$-equivariant inclusion
\begin{subequations} 
\begin{equation}
r_\lambda : \hoa \lambda \hookrightarrow k[\Uma]
\end{equation}
given for $v \in \hoa \lambda$ and $u \in \Uma$ by
\begin{equation} \label{eq:r explicitly}
r_\lambda(v)(u) = \etagr \lambda\big(  v \otimes u.\vps  \big)
\end{equation}
\end{subequations}
Although the statement is over $\C$ in \cite{KuNil} the proof is valid over any field. (Note that the statement preceeding Lemma 4.2 in \cite{KuNil} is not valid in arbitrary characteristic, but we do not need that statement to prove Lemma 4.2 in this special case since that lemma follows immediately from the definition of the filtration $\hof n \lambda$.) It is easy to check that $r_\lambda$ is $\Uma$-equivariant.

Next, the construction in \cite{KuNil} shows that for all $n \geq 0$ there is a commutative diagram
\begin{equation*}
\xymatrix{
\hoa \lambda \ar[d] \ar@{^(->}[r]^{\hspace{.2in} r_\lambda} & \cone \ar[d] \\
\hoa{n \lambda} \ar@{^(->}[r]_{\hspace{.2in} r_{n\lambda}} & \cone
}
\end{equation*}
where the vertical arrows are the $n^{th}$ power morphisms. Thus the maps $r_{n\lambda}$ for $n \geq 0$ induce a $\Uma$-equivariant algebra inclusion
\begin{align}
\r : \Ra \lambda &\hookrightarrow k[\Uma] \otimes k[t], \nonumber \\
\r (v) &= r_{n\lambda} (v) \otimes t^n \tr{ for all } v \in \hoa {n\lambda}.
\end{align}

(2) \tb{The map $\s$}

Identify $\Uma$ with the degenerate big cell $\Uma.[\vstar \lambda] \subseteq \flagastar \lambda$. We construct the map $\s$ by choosing a particular trivialization of $ \div 1 $ over $\Uma$.

There is a canonical nonvanishing section $c$ of the tautological bundle $\div{-1}|_\Uma$ given by $c(u) = u.\vstar \lambda$ for all $u \in \Uma$. Dualizing this section we obtain a trivialization
\begin{equation} \label{eq:s}
s : \Gamma\big(  \Uma, \div 1|_\Uma  \big) \isom k[\Uma]
\end{equation}
of $ \div 1|_\Uma $. Taking powers of this trivialization we obtain trivializations
$$ s_n : \Gamma\big(  \Uma, \div n|_\Uma  \big) \isom k[\Uma] $$ 
for all $n \geq 0$ such that $ s_n(f)^m = s_{mn}(f^m) $ for all $m \geq 0$ and $f \in \Gamma\big(  \Uma, \div n|_\Uma  \big)$. As a result we obtain an algebra inclusion
\begin{equation} \label{eq:Sa}
\s : \Sa \lambda \hookrightarrow k[\Uma] \otimes k[t]
\end{equation}
given by
$$  \s( f ) = s_n (f) \otimes t^n  $$
for all $n \geq 0$ and $f \in \Gamma\big(  \flagastar \lambda, \div n  \big) \subseteq \Gamma\big(  \Uma, \div n|_\Uma  \big)$.

(3) \tb{The commutativity of the diagram (\ref{eq:r and s})}

To show that (\ref{eq:r and s}) commutes it suffices to check that $r_\lambda = s \circ i$. By construction, the trivialization $s$ takes a section $v \in \Gamma(\Uma, \div 1|_\Uma)$ to its pairing with the element $c \in \Gamma(\Uma, \div{-1}|_\Uma)$. For $v$ in the image of $i : \hoa \lambda \hookrightarrow \Gamma(\Uma, \div 1|_\Uma)$ this pairing is given by the duality pairing
$$\etagr \lambda( v \otimes c(u)) = \etagr \lambda( v \otimes u.\vastar \lambda )$$
for $u \in \Uma$ and the result follows, cf (\ref{eq:r explicitly}).
\end{proof}

\begin{remarks}
\begin{enumerate}[(1)]
\item As a result of this proposition we see that $\Ra \lambda$ is a domain.
\item We have the following alternate geometric description of $\s$. Let $\linv$ denote the total space of the dual bundle $ \mc O(-1) $ on $\flagastar \lambda$ and let $\huma \subseteq \linv$ be the fiber over $\Uma$, an affine open subset of $\linv$. Then we have $ H^0( \linv, \mc O_\linv ) = \Sa \lambda $ and the trivialization $s$ of (\ref{eq:s}) gives an identification $ k[\huma] \cong \kuma \otimes k[t] $. With this setup, $\s$ identifies with the natural restriction map $H^0( \linv, \mc O_\linv ) \hookrightarrow k[\huma]$.
\end{enumerate}  
\end{remarks}

\subsubsection{$\Ra \lambda$ and $\Sa \lambda$} \label{subsub:Ra and Sa 2}

We now have the following fact which allows us under certain circumstances to identify $\proj (\Ra \lambda)$ with $\flagastar \lambda$.

\begin{proposition} \label{pr:proj Ra}
Let $\lambda \in \dom$ and assume that $ \Ra \lambda $ is generated as a $k$-algebra by its degree-1 subspace $\hoa \lambda$. Then $ \proj \,(\Ra \lambda) \cong  \proj \,(\Sa \lambda) = \flagastar \lambda $.
\end{proposition}

\begin{proof}
Since $\Ra \lambda$ is generated in degree 1 by its subspace $\hoa \lambda$, by Proposition \ref{pr:Ra and Sa} we have an inclusion $\Ra \lambda \hookrightarrow \Sa \lambda$ as subalgebras of $k[\Uma] \otimes k[t]$. Let $\mc J \subseteq \mc O_{\Pastar \lambda}$ be the ideal defining $\flagastar \lambda$. Then for all $n \geq 0$ we have an exact sequence
$$  0 \to \mc J(n) \to \opa \lambda n \to \ofa \lambda n \to 0  $$
of sheaves on $\Pastar \lambda$. For $n \gg 0$ this gives a surjection
\begin{equation} \label{eq:n gg 0}
\Gamma\big(  \Pastar \lambda, \opa \lambda n  \big) \twoheadrightarrow \Gamma\big( \flagastar \lambda, \ofa \lambda n   \big).
\end{equation}
Since $ \Gamma\big(  \Pastar \lambda, \opa \lambda 1  \big) = \hoa \lambda $, we have
$$ \Gamma\big(  \Pastar \lambda, \opa \lambda n  \big) = S^n \hoa \lambda $$
and thus by (\ref{eq:n gg 0}) there is a surjective map
\begin{equation} \label{eq:n gg 0 2}
S^n \hoa \lambda \twoheadrightarrow \Gamma\big( \flagastar \lambda, \ofa \lambda n   \big)
\end{equation}
for $n \gg 0$. Considering $ \hoa \lambda $ as a subspace of $\Sa \lambda$ via the section restriction map $i_\lambda$ of Proposition \ref{pr:Ra and Sa}, we see that (\ref{eq:n gg 0 2}) is given by multiplication in $\Sa \lambda$. Hence the inclusion $\Ra \lambda \hookrightarrow \Sa \lambda$ is an equality in high enough graded degrees and the desired isomorphism follows by Exercise 2.14(c) in \cite{Ha}.
\end{proof}

\begin{remark}
More generally, let us relax the assumption that $ \Ra \lambda $ is generated in degree 1. Let $\Ra \lambda' \subseteq \Ra \lambda$ be the $k$-subalgebra generated by the degree-1 subspace $\hoa \lambda$. Then the proof of Proposition \ref{pr:proj Ra} shows that there is an isomorphism $\proj(\Ra \lambda') \cong \flagastar \lambda$ in all types.
\end{remark}

\subsubsection{Frobenius splitting}

Recall that the existence of a Frobenius splitting of $\flag$ that maximally compatibly splits the identity is known in all classical types and in type $G_2$ by \cite{LMP}. As a result of Lemma \ref{lem:projective splitting}, Theorem \ref{th:Frobenius splitting of Ra} and Proposition \ref{pr:proj Ra} we have the following sharpening of Theorem \ref{th:ACG splitting} which gives a general criterion for a degenerate flag variety to be Frobenius split:

\begin{theorem} \label{th:splitting of flaga}
Choose $\lambda \in \dom$. Assume that $ \Ra {\lambda^*} $ is generated as a $k$-algebra by its degree-1 elements and that $\flag$ has a Frobenius splitting that compatibly splits the identity with maximum multiplicity. Then the degenerate flag variety $\flaga \lambda$ is Frobenius split.
\end{theorem}

\newpage
\bibliographystyle{amsplain}
\bibliography{/Users/Charles/Documents/TeX/thebibliography}

\end{document}